\documentclass[11pt]{article}
\usepackage{amsfonts,color}
\usepackage{amssymb}
\usepackage{amsmath}
\usepackage{amsthm}
\usepackage[square,sort,comma,numbers]{natbib}
\usepackage{enumerate}
\usepackage{graphicx}
\usepackage{wrapfig}
\usepackage{rotating}
\usepackage{epsfig}
\usepackage[left=1in, right=1in, top=0.8in, bottom=1in]{geometry}
\usepackage{mdwlist}
\usepackage{secdot}
\usepackage{comment}
\usepackage{mathrsfs}
\usepackage{multirow}
\usepackage{mdwlist}
\usepackage{multirow}
\usepackage{algorithm}
\usepackage{lscape}
\usepackage{pdflscape}
\usepackage{array}
\usepackage[compact]{titlesec}
\usepackage{times}
\usepackage{tikz}
\usepackage{subfigure}
\usepackage{relsize}
\usepackage[justification=centering]{caption}
\usepackage{setspace}
\usepackage{color}
\usepackage{authblk}
\usepackage{wrapfig}
\usepackage{hyperref}
\usepackage{makeidx}

\bibpunct{[}{]}{,}{n}{}{;}
\titlespacing*{\section}{0pt}{*2}{*0}
\titlespacing*{\subsection}{0pt}{*2}{*0}
\titlespacing*{\subsubsection}{0pt}{*1}{*0}

\newtheorem{theorem}{Theorem}
\newtheorem{lemma}{Lemma}

\newtheorem{proposition}{Proposition}
\newtheorem{corollary}{Corollary}

\newtheorem{definition}{Definition}
\newtheorem{problem}{Problem}

\newtheoremstyle{compacttheorem}
{3pt}
{3pt}
{}
{}
{}
{}
{}
{}

\newtheoremstyle{break}
{\topsep}{\topsep}%
{\itshape}{}%
{\bfseries}{}%
{\newline}{}%
\theoremstyle{break}

\title{A Bounded Formulation for The School Bus Scheduling Problem}
\author[1]{\rm Liwei Zeng}
\author[2]{\rm Sunil Chopra}
\author[1]{\rm Karen Smilowitz}
\affil[1]{Department of Industrial Engineering and Management Sciences}
\affil[2]{Kellogg School of Management}
\affil[ ]{Northwestern University}

\begin{document}
\maketitle
\begin{abstract}
This paper proposes a new formulation for the school bus scheduling problem (SBSP) which optimizes school start times and bus operation times to minimize transportation cost. Our goal is to minimize the number of buses to serve all bus routes such that each route arrives in a time window before school starts. We present a new time-indexed integer linear programming (ILP) formulation for this problem. Based on a strengthened version of the linear relaxation of the ILP, we develop a dependent randomized rounding algorithm that yields near-optimal solutions for large-scale problem instances. We also generalize our methodologies to solve a robust version of the SBSP.
\end{abstract}


\section{Introduction} \label{sec:intro}

Across the United States, public school districts are facing critical budget challenges (\citet{bidwell2015}) and transportation is a common area to look for cost reduction. Many school districts have changed school start times to lower transportation cost. By staggering the start times of different schools, a bus can be re-used to complete more than one route. The marginal cost of re-using a bus for a second route is significantly lower than that of adding a new bus. In this paper, we study the problem of jointly determining school start times and bus route operation times in a way that minimizes the number of buses needed to complete all bus routes.

The school bus routing problem (SBRP) has been extensively studied in the literature (\citet{newton1969design}, \citet{desrosiers1980overview}, \citet{park2010school}). The SBRP consists of two groups of subproblems: routing and scheduling. The routing subproblems include selecting bus stops, assigning students to stops, and connecting stops into bus routes. This paper focuses on the scheduling subproblems, which determine school start times, bus route operation times, and route-to-bus assignment. We use the term ``school bus scheduling problem'' (SBSP) from \citet{fugenschuh2009solving} to represent the joint problem of determining school start times, bus route operation times, and route-to-bus assignment.

\begin{problem}{\textbf{(School bus scheduling problem)}}\label{def:pscst}
	Given a set of schools and a set of routes for each school, the school bus scheduling problem determines the start times of schools and the arrival times of routes such that each route arrives in a time window before school starts. After school start times and route arrival times are determined, routes are assigned to buses such that routes assigned to the same bus operate on disjoint time intervals. The goal is to minimize the number of buses needed to complete all the routes.
\end{problem}

The SBSP is NP-hard, which can be proved by a reduction from the balanced partition problem (\citet{garey2002computers}). Previous work has identified both exact and inexact algorithms for the SBSP, such as column generation (\citet{desrosiers1986transcol}), branch-and-cut (\citet{fugenschuh2009solving}) and local search (\citet{bertsimas2019optimizing}). While the exact algorithms perform well on small instances, they encounter computational challenges as the problem scale increases, which can become an issue when being applied to large school districts. Inexact algorithms scale well but typically lack theoretical guarantees.

In this work, we propose a new approach to solve SBSP that scales well and has a provable performance guarantee. Having a fast solution technique with a provable performance guarantee allows decision makers to more easily incorporate the impact on transportation cost into various strategic planning decisions. Our approach also provides decision makers with a set of low transportation cost solutions that can be used to select one that best satisfies a variety of other requirements. As a result, our work aims to support strategic planning within school districts in contrast to other work that may be more operationally focused and designed to provide immediately actionable plans. We formulate the SBSP as a time-indexed ILP where binary variables are used to indicate school start times and route arrival times. There are several papers (\citet{lenstra1990approximation}, \citet{shmoys1993approximation}, \citet{schulz1996scheduling}) in the scheduling literature that have used LP relaxations of suitable ILP formulations to develop efficient approximation algorithms with good lower bounds. We join this line of work by first strengthening the time-indexed ILP through characterizing the convex hulls of variable subspaces. The strengthened formulation is shown to have a bounded integrality gap. We then develop a randomized rounding algorithm which is near-optimal for large-scale instances. The randomized nature of the algorithm provides more flexibility for decision makers, allowing them to select from multiple high-quality solutions, such that additional factors may be considered, see \cite{banerjee2019incorporating} which includes equity in determining school start times.

The main contribution of this paper is as follows:

\begin{enumerate}
	
	\item We develop a new time-indexed ILP formulation for the SBSP. Based on the LP relaxation of the ILP, we propose a randomized rounding algorithm with a bounded performance guarantee that provides multiple high-quality solutions to the SBSP efficiently. Our approach enables school districts to quickly estimate the impact on transportation cost of different plans when considering school start time changes. The multiple high quality solutions provided by the randomized rounding algorithm give decision makers more options to incorporate external considerations while maintaining low transportation cost.
	
	\item We are able to solve small-scale instances of SBSP to optimality using the new ILP formulation. For large-scale instances, we prove that the randomized rounding algorithm is near-optimal. The near-optimality of our algorithm provides accurate estimation of potential cost-saving of different policies, which allows decision makers to focus on plans with sufficient potential for cost reduction. In an on-going collaboration with a public school district, such high quality and efficient solution approaches are critical for the evaluation of scheduling strategies.
	
	\item We introduce a robust version of the SBSP. Robustness is a critical issue in school bus scheduling because school start times cannot be redesigned frequently, yet bus routes may vary from year to year (due to changes in demographic distribution, traffic conditions, etc.). From a strategic planning perspective, school districts are interested in school start time plans that will remain consistent over time and still lower transportation cost. We generalize the ILP model and solution approaches to the SBSP and provide a similar randomized rounding algorithm for the robust SBSP with provable performance guarantee.
\end{enumerate}

The remainder of this paper is organized as follows. In Section \ref{sec:literature} we review related work on school bus scheduling, machine scheduling and bin packing. In Section \ref{sec:ILP} we present a time-indexed ILP for the SBSP, which is strengthened with valid cuts. In Section \ref{sec:rounding}, we develop a randomized rounding algorithm for the SBSP based on the strengthened LP-relaxation that is provably near-optimal for large-scale problem instances. Our methodology is generalized in Section \ref{sec:robust} for  a robust version of the SBSP. In Section \ref{sec:numerical}, we present computational results that complement our theoretical findings. We conclude in Section \ref{sec:conclusion} with a summary of results and discussion of future research.


\section{Literature Review} \label{sec:literature}
We summarize related literature in school bus scheduling. We also link the SBSP to two fundamental combinatorial optimization problems: the machine scheduling problem and the bin packing problem.

\subsection{School Bus Scheduling Problem}\label{subsec:literature-school-bus}
As a part of the school bus routing problem, the SBSP specifies school start times and route arrival times to maximize bus usage. \citet{swersey1984scheduling} present a mixed-integer programming (MIP) formulation to determine route arrival times when school start times are fixed. The MIP formulation is further simplified to an integer programming (IP) formulation by discretizing the timeline into unit intervals. \citet{desrosiers1986transcol} determine school start times and route arrival times sequentially. They formulate the problem as a min-max binary program that minimizes the maximum number of routes operating during the same time period. The formulation is solved using column generation for small-sized instances. For large-scale instances, they present a heuristic that alternately updates school start times and route arrival times. \citet{fugenschuh2009solving} presents an IP formulation that determines school start times and route arrival times simultaneously and solves it using a branch-and-cut algorithm. \cite{banerjee2019incorporating} uses a variant of the time-indexed formulation developed in this paper to consider equity (defined by the disutility associated with changing school start times) in the SBSP.

The problem of jointly optimizing bus routes and school bell times has also received attention in recent literature. \citet{spada2005decision} discuss the combined problem of bus route generation and bus scheduling where they provide an initial feasible solution by solving an IP, which is then improved using three different heuristics. \citet{bertsimas2019optimizing} construct a set of routing scenarios for each school using a  ``bi-objective routing decomposition algorithm'', and then formulate the school bell selection problem as a generalized quadratic assignment problem. They implement their algorithms in collaboration with the Boston Public Schools, which has led to \$5 million saving per year.

In contrast to previous work, we focus on providing an efficient tool box with provable performance guarantees to support strategic decision-making for school districts. We also extend our approach to help school districts identify solutions that are robust to external changes.

\subsection{Machine Scheduling and Bin Packing}\label{subsec:literature-maching-scheduling}
Our work is related to machine scheduling problems that consider job priorities. In these problems, jobs with higher priority must arrive (or end) earlier than others. \citet{ikura1986efficient} study the batched scheduling problem where jobs in the same batch have the same priority. A simplified version of the SBSP (where each route arrives exactly at the school's start time) can be restated as the batched machine scheduling problem where routes in the same school have the same priority. The SBSP is also similar in spirit to the 1-dimensional bin packing problem (\citet{de1981bin}, \citet{scholl1997bison}) where objects are bus routes, and bins are buses. The constraints on route arrival times can be transformed to constraints on objects' relative locations.


\section{Integer Linear Programming Formulation for the SBSP}\label{sec:ILP}

We first introduce a time-indexed formulation for the SBSP. We then strengthen this formulation by identifying a semi-decomposable structure of the first formulation and characterizing the convex hulls of variable subspaces. The optimal fractional solution of the strengthened formulation is used in Section \ref{sec:rounding} to develop a randomized rounding algorithm.

\subsection{Preliminaries}\label{subsec:transition}
We use the following notation throughout the paper. Let $\mathcal{S}$ be a set of schools and let $\mathcal{R}=\bigcup_{s\in\mathcal{S}}\mathcal{R}_s$ be a set of routes, where $\mathcal{R}_s$ is the set of routes for school $s$. We discretize the timeline into $T$ unit intervals and assume that each school (route) chooses a start (arrival) time from $[T]=\{1,2,\ldots,T\}$. For each school $s$, $l_s\in \mathbb{N}$ represents the length of its time window, in which routes for this school must arrive. Specifically, if school $s$ starts at time $t_s$, all routes for this school must arrive in the interval $[t_{s}-l_{s},t_{s}]$ ($[1,t_s]$ if $t_s-l_s\leq 0$). Finally, we use $r_i\in \mathbb{N}$ to denote the length of route $i$.

When school districts are compact, one can make the assumption of constant transit time between routes (this assumption is reasonable in our collaboration where the range of bus transition times is significantly smaller than route travel and stopping times). Importantly, this assumption allows us to leverage more structural information and design fast algorithms with provable bounds.  To generalize our results, we present methods and numerical studies to apply our approach to  settings where this assumption does not hold.

In the following discussion, the constant transition time is added to the beginning of each route (e.g. a 30-minute route with 10-minute transition becomes a 40-minute route) and we assume that there is no transition time between routes. We note that, in practice, the transition time should not be added to the first route taken by each bus. However, we claim that adding the same amount of time to each route (including the first route taken by each bus) will not affect the optimal solution. The time added to the first route can be viewed as the travel time from the depot to the first stop of the first route, which does not affect the route connection pattern.

\subsection{A Time-indexed Formulation for the SBSP}\label{subsec:time-index}
Time-indexed formulations have been widely applied to scheduling problems to obtain strong lower bounds (\citet{dyer1990formulating}, \citet{sousa1992time}, \citet{queyranne1994polyhedral}. In these formulations, the time horizon is divided into unit time periods, and binary variables are used to indicate start/end times of jobs. 

We use a similar approach to develop a time-indexed formulation for the SBSP. Our formulation is similar to the one in \citet{desrosiers1986transcol}, where binary variables are used to indicate both school start times and route arrival times. The novelty of our formulation is a set of constraints that captures the relation between school start times and route arrival times.

We start with the following observation in \citet{desrosiers1986transcol} that transforms the minimum number of buses to an easy-to-compute quantity.

\begin{proposition}[\citet{desrosiers1986transcol}]\label{prop:numroute}
	Given a set of routes with fixed operation times, the minimum number of buses to complete all the routes is equal to the maximum number of routes in operation during the same time period.
\end{proposition}

Given the arrival time of each route, Proposition \ref{prop:numroute} states that the minimum number of buses required is easy to compute. Specifically, we can divide the timeline into unit intervals and count the number of routes in operation during each interval. From Proposition \ref{prop:numroute}, the minimum number of buses required to serve all routes is equal to the maximum number recorded over all intervals. Moreover, with fixed route operation times, the route-to-bus assignment problem is equivalent to an interval graph coloring problem, which can be solved in polynomial time (see \citet{olariu1991optimal} for details). Motivated by this observation, we introduce a time-indexed formulation in which the maximum number of routes in operations during the same unit interval is minimized.

For each route $i\in \mathcal{R}$ and time $t\in [T]$, let $x_{i,t}$ be a binary variable such that $x_{i,t}=1$ if route $i$ arrives at time $t$. Similarly, for each school $s\in \mathcal{S}$ and time $t\in [T]$, we introduce a binary variable $y_{s,t}$ such that $y_{s,t}=1$ if school $s$ starts at time $t$. The variables are defined based on route arrival times and school start times for notational convenience. We extend the definition of $x,y$ variables to $t\notin [T]$ and define that $x_{i,t}=0, y_{i,t}=0$ if $t\notin [T]$.

The SBSP can be formulated as follows (formulation ILP1).

\begin{align} 
	\textrm{min}~~&z \label{ILP1} \tag{ILP1} \\
	\textrm{s.t.}~~&\sum_{t=1}^{T} x_{i,t}=1 & \forall i\in \mathcal{R}_s, s\in \mathcal{S} \label{eqn:1a} \tag{1a} \\
	&\sum_{t=1}^{T}y_{s,t}=1 &\forall s\in \mathcal{S} \label{eqn:1b} \tag{1b} \\
	&x_{i,t}\leq \sum_{t^{'}=t}^{t+l_s} y_{s,t^{'} } &\forall i\in \mathcal{R}_s, s\in \mathcal{S}, t\in [T] \label{eqn:1c} \tag{1c} \\
	&\sum_{i\in \mathcal{R} }\sum_{t^{'}=t}^{t+r_i-1} x_{i,t^{'} }\leq z &\forall t\in [T] \label{eqn:1d} \tag{1d} \\
	&x_{i,t}\in \{0,1\} &\forall i\in \mathcal{R}_s, s\in \mathcal{S}, t\in [T] \label{eqn:1e} \tag{1e}\\
	&y_{s,t}\in \{0,1\} &\forall s\in \mathcal{S}, t\in [T] \label{eqn:1f} \tag{1f}
\end{align}

Assignment constraints \eqref{eqn:1a} and \eqref{eqn:1b} ensure that each route (school) is assigned to one arrival (start) time. Time-window constraints \eqref{eqn:1c} enforce that each route arrives in a time window before school starts. Constraints \eqref{eqn:1d} link the decision variables to the objective by Proposition \ref{prop:numroute}. For each $t\in [T]$, the left-hand side of \eqref{eqn:1d} computes the number of routes in operation during unit interval $[t-1,t]$. By introducing a constraint for each $t\in [T]$, \eqref{eqn:1d} implies that the objective $z$ is equal to the maximum number of routes in operation during the same unit time interval.

\ref{ILP1} can also incorporate various types of restrictions on school start times and route operation times by fixing part of the decision variables. As shown later, this does not affect the theoretical performance guarantee of the algorithm.

Although the time-indexed formulation is known to provide better bounds, it is hard to solve directly due to size. The number of binary variables in \ref{ILP1} is $(|\mathcal{R}|+|\mathcal{S}|)T$. For a district with 20 schools, 100 routes and $T=120$ (2-hour time horizon divided into 1-minute intervals), this requires solving an ILP with 14,400 variables (and even more constraints). Thus, we first strengthen the LP-relaxation of ILP1 by adding valid cuts.

\subsection{A Strong Integer Linear Programming Formulation}\label{subsec:strong-formulation}

\begin{figure}
	\centering
	\includegraphics[width=0.7\textwidth]{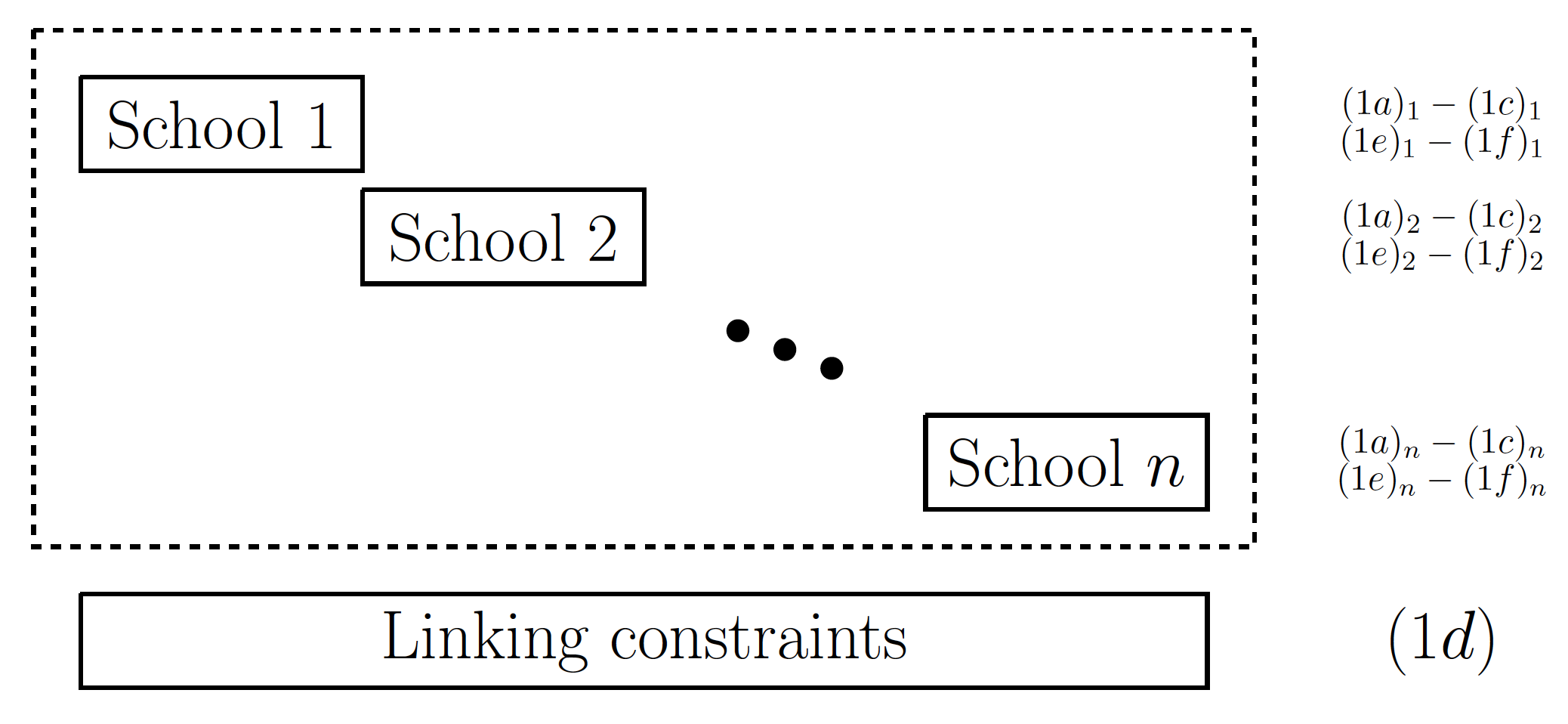}
	\caption{Block-structure of \ref{ILP1}}
	\label{fig:block}
\end{figure}

We present a strengthened formulation of \ref{ILP1}. The LP relaxation of the strengthened formulation is used in Section \ref{sec:rounding} to design a randomized rounding algorithm that solves large-scale instances to near-optimality. We also show in Section \ref{sec:numerical} that the strengthened formulation is able to provide tight lower bounds in practice through a numerical study.

Observe that \ref{ILP1} has a semi-decomposable structure in that all its constraints, except for \eqref{eqn:1d}, only contain variables for a single school. By suitably reordering the variables and constraints, \ref{ILP1} exhibits the block-structure shown in Figure \ref{fig:block}.

For each school$s\in \mathcal{S}$ we define
\[V_s=\big\{(x_{i,t},y_{s,t})~\big |~i\in \mathcal{R}_s, (x_{i,t},y_{s,t})~\textrm{satisfies \eqref{eqn:1a},\eqref{eqn:1b},\eqref{eqn:1c},\eqref{eqn:1e} and \eqref{eqn:1f}} \big\}. \]
$V_s$ can be interpreted as the set of variables satisfying the block of constraints for school $s$ in Figure \ref{fig:block}. Hence, \ref{ILP1} can be rewritten as
\begin{align*} 
	\textrm{min}~~&z\\
	\textrm{s.t.}~~&\sum_{i\in \mathcal{R} }\sum_{t^{'}=t}^{t+r_i-1} x_{i,t^{'} }\leq z &\forall t\in [T]\\
	&(x_{i,t},y_{s,t})\in V_{s} &\forall s\in \mathcal{S}
\end{align*}

A natural way to strengthen the LP-relaxation of \ref{ILP1} is by completely defining Conv$(V_{s})$, the convex hull of $V_{s}$. The following theorem provides a complete characterization of Conv$(V_{s})$.
\begin{theorem}[Convex Hull of $V_{s}$]\label{thm:convex-hull}
	For any $s\in \mathcal{S}$, $Conv(V_s)=C_s$, where $C_{s}$ is defined by the following constraints:
	\begin{enumerate}[label=(C).]
		\item $\sum_{t=1}^{T} x_{i,t}=1, \forall i\in \mathcal{R}_s$
		\item $\sum_{t=1}^{T} y_{s,t}=1$
		\item $\sum_{t^{'}=1}^{t} x_{i,t^{'} }\leq \sum_{t^{'}=1}^{t+l_s} y_{s,t^{'} }, \forall i\in \mathcal{R}_s, t\in [T]$
		\item $\sum_{t^{'}=1}^{t} y_{s,t^{'} }\leq \sum_{t^{'}=1}^{t} x_{i,t^{'} }, \forall i\in \mathcal{R}_s, t\in [T]$
		\item $0\leq x_{i,t} \leq 1, \forall i\in \mathcal{R}_s, t\in [T]$
		\item $0\leq y_{s,t} \leq 1, \forall t\in [T]$.
	\end{enumerate}
\end{theorem}

$(C1)-(C2)$ are the assignment constraints, and $(C5)-(C6)$ are relaxations of the integrality constraints. $(C3)-(C4)$ combined can be viewed as a stronger form of the time-window constraint \eqref{eqn:1c}. Specifically, for any school $s$ and route $i\in \mathcal{R}_s$, $(C3)$ implies that the start time of $s$ is no later than the arrival time of $i$ plus a time-window length $l_s$. $(C4)$ implies that the arrival time of $i$ is no later than the start time of $s$.

\textbf{Proof of Theorem \ref{thm:convex-hull}.}\label{proof:convexhull}
We first prove that $Conv(V_s)\subseteq C_{s}, \forall s\in \mathcal{S}$. Since the set of extreme points of $Conv(V_s)$ is exactly $V_s$, it suffices to show that $V_{s}\subseteq C_{s}$. This is equivalent to saying that $(x_s,y_s)$ satisfies $(C1)-(C6)$ for any $(x_s,y_s)\in V_s$. It is easy to see that $(C1),(C2),(C5)$ and $(C6)$ follow from the assignment constraints \eqref{eqn:1a}-\eqref{eqn:1b}, and the binary constraints \eqref{eqn:1e}-\eqref{eqn:1f}. It remains to show that $(x_s,y_s)$ satisfies $(C3)$ and $(C4)$ as well.

For any $(x_s,y_s)\in V_s$ and $i\in \mathcal{R}_{s}$, let $t_i,t_s\in [T]$ be indices such that $x_{i,t_i}=1$ and $y_{s,t_s}=1$. In constraint \eqref{eqn:1c}, let $t=t_i$. We have $1\leq \sum_{t^{'}=t_i}^{t_i+l_s} y_{s,t^{'} }$, which implies that $t_i\leq t_s\leq t_i+l_s$. Note that $(C3)$ and $(C4)$ can be derived from $t_s\leq t_i+l_s$ and $t_i\leq t_s$, respectively. Therefore, $(x_s,y_s)$ satisfies $(C1)-(C6)$.

We next prove that $C_{s}\subseteq Conv(V_s), \forall s\in \mathcal{S}$. As a first step, we show that every integral point in $C_{s}$ is also in $V_s$. To see that, for any integral $(x_s,y_s)\in C_{s}$ and $i\in \mathcal{R}_{s}$, let $t_i,t_s\in [T]$ be indices such that $x_{i,t_i}=1$ and $y_{s,t_s}=1$. Similar to the statement in the last paragraph, $(C3)$ and $(C4)$ imply that $t_s\leq t_i+l_s$ and $t_i\leq t_s$, which combined indicate that $(x_s,y_s)$ satisfies \eqref{eqn:1c}. Besides, since \eqref{eqn:1a}, \eqref{eqn:1b}, \eqref{eqn:1e} and \eqref{eqn:1f} naturally follow from $(C1), (C2), (C5)$ and $(C6)$, we have $(x_s,y_s)\in V_s$.

Now consider the linear projection $f: C_{s}\rightarrow P_s$ defined by
\[ f(x_s,y_s)=(X_s,Y_s)~\textrm{if}~X_{i,t}=\sum_{t^{'}=1}^{t} x_{i,t^{'} }~\textrm{and}~Y_{s,t}=\sum_{t^{'}=1}^{t} y_{s,t^{'} }, \]
where $P_s$ is the polytope defined by the following linear inequalities:
\begin{enumerate}[label=(P)]
	\item $X_{i,T}=1, \forall i\in \mathcal{R}_s$
	\item $Y_{s,T}=1$
	\item $X_{i,t}\leq Y_{s,\textrm{min}\{t+l_s,T\} }, \forall i\in \mathcal{R}_s, t\in [T]$
	\item $Y_{s,t}\leq X_{i,t}, \forall i\in \mathcal{R}_s, t\in [T]$
	\item $X_{i,t}\leq X_{i,t+1}, \forall i\in \mathcal{R}_s, t\in [T-1]$
	\item $Y_{s,t}\leq Y_{s,t+1}, \forall t\in [T-1]$.
\end{enumerate}

The linear constraints $(P1)-(P6)$ correspond to $(C1)-(C6)$ by transforming $(x,y)$ variables to $(X,Y)$ variables. It is easy to verify that the inverse projection $f^{-1}: P_s\rightarrow C_{s}$ exists, and is linear:
\begin{equation}\label{eqn:f-inverse}
	f^{-1}(X_s,Y_s)=(x_s,y_s)~\textrm{if}~x_{i,t}=X_{i,t}-X_{i,t-1 }~\textrm{and}~y_{s,t}=Y_{s,t}-Y_{s,t-1}.
\end{equation}
Here we define $X_{i,0}=0$, $Y_{s,0}=0$.

Moreover, the matrix containing $(P1)-(P6)$ is totally unimodular since it has no more than two non-zero entries in each column (constraint) and every column (constraint) that contains two non-zero entries has exactly one $+1$ and one $-1$ (\cite{nemhauser1988integer}). Since the matrix defined by $(P1)-(P6)$ is totally unimodular, all extreme points of $P_s$ are integral.

For any $(\tilde{x}_s,\tilde{y}_s)\in C_{s}$, note that $f(\tilde{x}_s,\tilde{y}_s)\in P_s$, by representing $f(\tilde{x}_s,\tilde{y}_s)$ as a linear combination of extreme points of $P_s$, we have
\begin{align*}
	(\tilde{x}_s,\tilde{y}_s)&=f^{-1}(f(\tilde{x}_s,\tilde{y}_s))\\
	&=f^{-1}(\sum_{k=1}^{K}\lambda_{k}p_{k})\\
	&=\sum_{k=1}^{K}\lambda_{k}f^{-1}(p_k)
\end{align*}
where each $\lambda_{k}$ is non-negative, $\sum_{k=1}^{K}\lambda_{k}=1$, and all $p_k$ are integral extreme points of $P_s$. The last step applies the linearity of $f^{-1}$. Figure \ref{fig:projection} visualizes each step of the projection.

\begin{figure}
	\centering
	\includegraphics[width=0.6\textwidth]{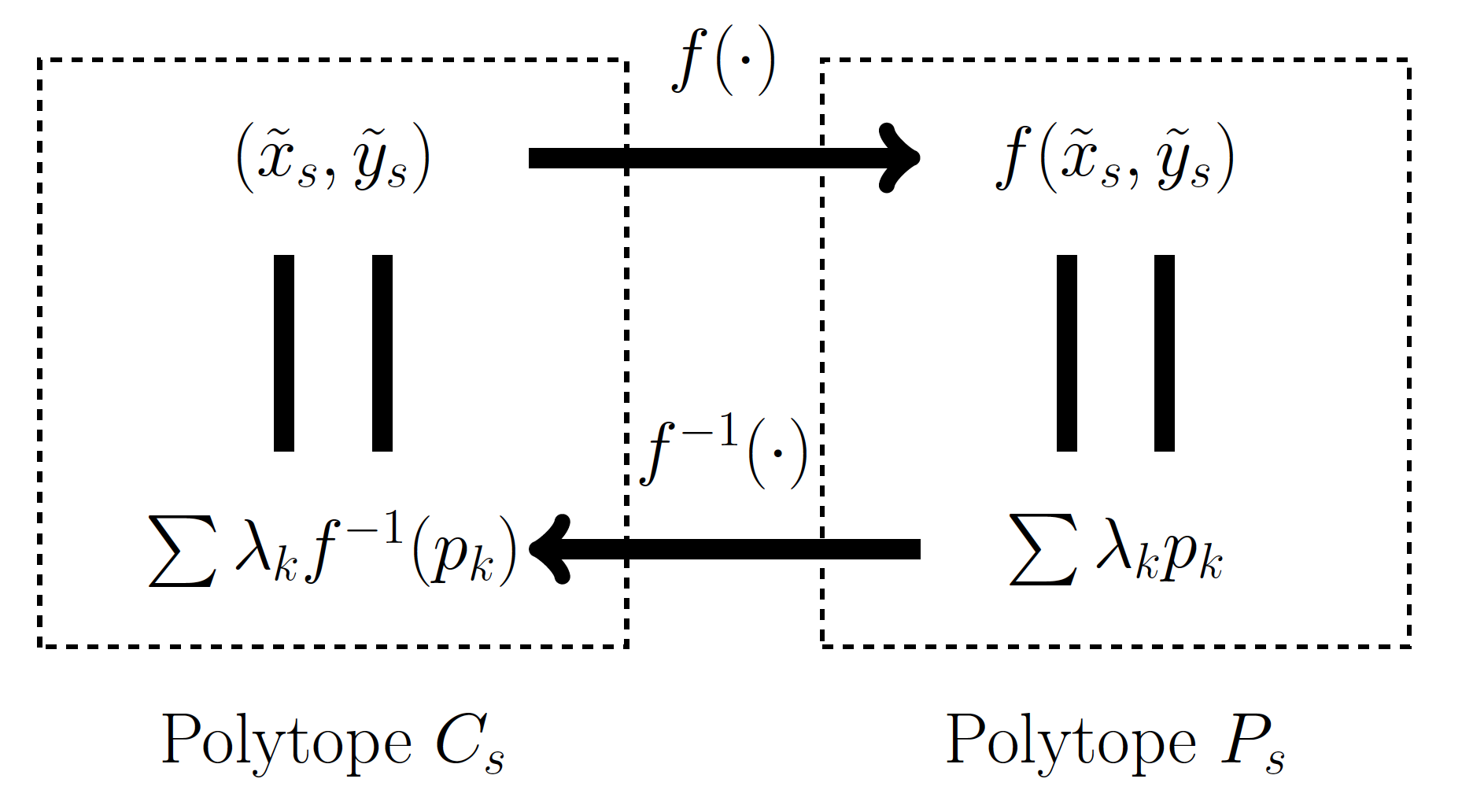}
	\caption{Projections between polytopes $C_s$ and $P_s$}
	\label{fig:projection}
\end{figure}

From equation \eqref{eqn:f-inverse}, $f^{-1}(p_k)\in C_{s}$ is also integral given that $p_k$ is integral for every $k$. This means $(\tilde{x}_s,\tilde{y}_s)$ can be written as a non-negative linear combination of integral points in $C_{s}$. Since these integral points are also in $V_s$, $(\tilde{x}_s,\tilde{y}_s)$ must be in the convex hull of $V_s$, which immediately implies $C_{s}\subseteq Conv(V_s)$. This completes the proof of Theorem \ref{thm:convex-hull}.
\hfill\ensuremath{\square}

Using Theorem \ref{thm:convex-hull}, we obtain the following formulation ILP2 for the SBSP by replacing \eqref{eqn:1c} in ILP1 with $(C3)$ and $(C4)$.
\begin{align} 
	\textrm{min}~~&z \label{ILP2} \tag{ILP2} \\
	\textrm{s.t.}~~&\sum_{t=1}^{T} x_{i,t}=1 & \forall i\in \mathcal{R}_s, s\in \mathcal{S} \label{eqn:2a} \tag{2a} \\
	&\sum_{t=1}^{T}y_{s,t}=1 &\forall s\in \mathcal{S} \label{eqn:2b} \tag{2b} \\
	&\sum_{t^{'}=1}^{t} x_{i,t^{'} }\leq \sum_{t^{'}=1}^{t+l_s} y_{s,t^{'} } &\forall i\in \mathcal{R}_s, s\in \mathcal{S}, t\in [T] \label{eqn:2c} \tag{2c} \\
	&\sum_{t^{'}=1}^{t} y_{s,t^{'} }\leq \sum_{t^{'}=1}^{t} x_{i,t^{'} } &\forall i\in \mathcal{R}_s, s\in \mathcal{S}, t\in [T] \label{eqn:2d} \tag{2d} \\
	&\sum_{i\in \mathcal{R} }\sum_{t^{'}=t}^{t+r_i-1} x_{i,t^{'} }\leq z &\forall t\in [T] \label{eqn:2e} \tag{2e} \\
	&x_{i,t}\in \{0,1\} &\forall i\in \mathcal{R}_s, s\in \mathcal{S}, t\in [T] \label{eqn:2f} \tag{2f}\\
	&y_{s,t}\in \{0,1\} &\forall s\in \mathcal{S}, t\in [T] \label{eqn:2g} \tag{2g}
\end{align}

In the next section we prove that the LP relaxation of \ref{ILP2} has a bounded integrality gap. 


\section{Randomized Rounding Algorithm for the SBSP}\label{sec:rounding}

We develop a dependent randomized rounding algorithm for the SBSP based on the LP relaxation of \ref{ILP2}. The algorithm constructs a feasible integral solution to \ref{ILP2} based on a fractional optimal solution of the LP relaxation of \ref{ILP2}. We show in Theorem \ref{thm:errorbound} that the optimality gap of the resulting integral solution can be bounded by the square root of the optimal solution. We show in Corollary \ref{coro:integrality-gap} that this also gives a bounded integrality gap of the LP relaxation of \ref{ILP2}.

\begin{algorithm}\label{algorithm:rounding_general}
	\caption{(Randomized Rounding Algorithm for the SBSP)}
	Step 1: solve the LP relaxation of \ref{ILP2} to get $\{(x^{*}_s,y^{*}_s)\}_{s\in \mathcal{S}}$
	
	Step 2: for each school $s\in \mathcal{S}$, draw a uniform random number $r_s\sim U[0,1]$
	
	Step 3: for each route $i\in \mathcal{R}_s$, let $t_i=\textrm{argmin}\{\sum_{t=1}^{t_i}x^{*}_{i,t}\geq r_s \}$ and set $x_{i,t_i}=1$
	
	Step 4: for each school $s\in \mathcal{S}$, let $t_s=\textrm{argmin}\{\sum_{t=1}^{t_s}y^{*}_{s,t}\geq r_s \}$ and set $y_{s,t_s}=1$
	
	Step 5: solve the route-to-bus assignment problem using the scheduling from Steps 3 and 4
\end{algorithm}

We illustrate the algorithm with an example of one school and two routes (see Figure \ref{fig:rounding}). For each school and route, we divide a $[0,1]$ segment into $T$ intervals based on the fractional optimal solution $\{(x^{*}_s,y^{*}_s)\}$. We then draw a random number $r_s\sim U[0,1]$ (Step 2) and cut all $[0,1]$ segments at $r_s$ (Steps 3 and 4). For each school (route), the index of the interval where the random number $r_s$ falls is selected as its start (arrival) time.

\begin{figure}
	\centering
	\includegraphics[width=0.7\textwidth]{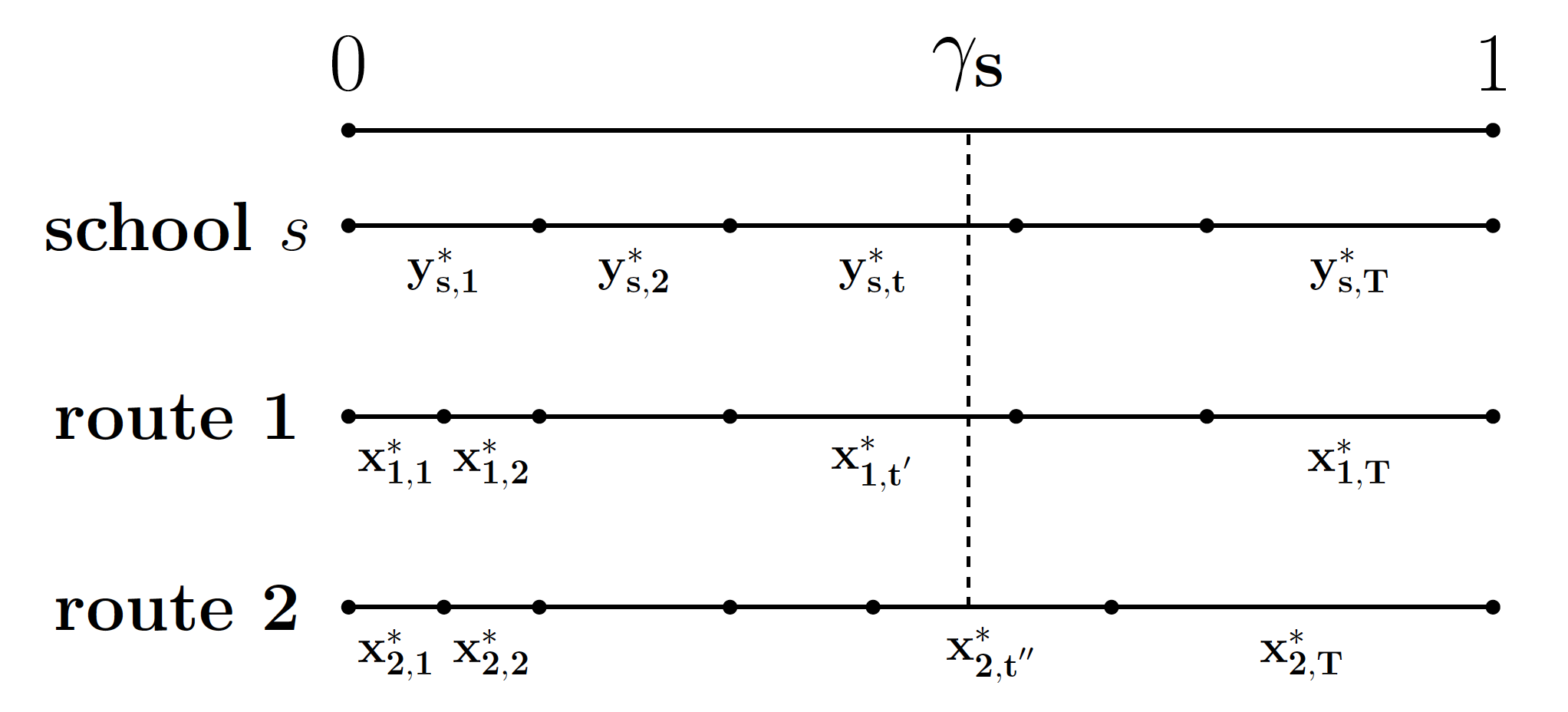}
	\caption{Randomized rounding algorithm}
	\label{fig:rounding}
\end{figure}

Let $\{(x_s,y_s)\}_{s\in \mathcal{S}}$ be the integral solution generated from Algorithm 1. In Proposition \ref{prop:correctness}, we show that the integral solution is always feasible. In Proposition \ref{prop:prob_property}, we show that the fractional optimal solution $\{(x^{*}_s,y^{*}_s)\}_{s\in \mathcal{S}}$ defines the probability that each corresponding variable takes the value 1 in the integer solution generated by Algorithm 1. Finally, we prove in Theorem \ref{thm:errorbound} that the integral solution is near-optimal with constant probability.

\begin{proposition}\label{prop:correctness}
	$\{(x_s,y_s)\}_{s\in \mathcal{S}}$ satisfies constraints \eqref{eqn:2a}-\eqref{eqn:2d} and \eqref{eqn:2f}-\eqref{eqn:2g}.
\end{proposition}

\textbf{Proof.}
In Steps 3 and 4 of Algorithm 1, each route (school) is assigned to exactly one arrival (start) time. Hence, the assignment constraints \eqref{eqn:2a}-\eqref{eqn:2b}, and the binary constraints \eqref{eqn:2f}-\eqref{eqn:2g} naturally hold. To see the time-window constraints \eqref{eqn:2c}-\eqref{eqn:2d} also hold, it suffices to show that 
\[t_i\leq t_s\leq t_i+l_s \]
for any school $s\in \mathcal{S}$ and route $i\in \mathcal{R}_s$.

Recall that both $t_i$ and $t_s$ are determined by the random variable $r_s$ generated in Step 2. From the definition of $t_i$ (Step 3) and $t_s$ (Step 4), we have
\begin{equation}
	\sum_{t=1}^{t_i-1}x^{*}_{i,t}< r_s\leq \sum_{t=1}^{t_i}x^{*}_{i,t}
\end{equation}
and
\begin{equation}
	\sum_{t=1}^{t_s-1}y^{*}_{s,t}< r_s\leq \sum_{t=1}^{t_s}y^{*}_{s,t}.
\end{equation}
Since the fractional solution $\{(x^{*}_s,y^{*}_s)\}_{s\in \mathcal{S}}$ satisfies the time-window constraints \eqref{eqn:2c}-\eqref{eqn:2d}, we have
\begin{equation}\label{eqn:prop2-1}
	\sum_{t=1}^{t_s-1}y^{*}_{s,t}<r_s\leq \sum_{t=1}^{t_i}x^{*}_{i,t}\leq \sum_{t=1}^{t_i+l_s}y^{*}_{s,t}.
\end{equation}
Comparing the left-hand side and the right-hand side of \eqref{eqn:prop2-1}, we have $t_s-1<t_i+l_s$, i.e., $t_s\leq t_i+l_s$.

Similarly, from
\begin{equation}\label{eqn:prop2-2}
	\sum_{t=1}^{t_i-1}x^{*}_{i,t}<r_s\leq \sum_{t=1}^{t_s}y^{*}_{s,t}\leq \sum_{t=1}^{t_s}x^{*}_{i,t}
\end{equation}
we have $t_i-1<t_s$, i.e, $t_s\geq t_i$. This completes the proof of Proposition \ref{prop:correctness}.
\hfill\ensuremath{\square}

\begin{proposition}\label{prop:prob_property}
	$\mathrm{Pr}[x_{i,t}=1]=x^{*}_{i,t}~\forall i\in \mathcal{R}, t\in [T]$; $\mathrm{Pr}[y_{s,t}=1]=y^{*}_{s,t}~\forall s\in \mathcal{S}, t\in [T]$.
\end{proposition}

\textbf{Proof.}
$x_{i,t}=1$ if and only if the random number $r_s$ falls into interval $[\sum_{t^{'}=1}^{t-1}x^{*}_{i,t^{'} },\sum_{t^{'}=1}^{t}x^{*}_{i,t^{'} }]$. Since $r_s\sim U[0,1]$,
\[\mathrm{Pr}[x_{i,t}=1]=\sum_{t^{'}=1}^{t}x^{*}_{i,t^{'} }-\sum_{t^{'}=1}^{t-1}x^{*}_{i,t^{'} }=x^{*}_{i,t}. \]
The proof for $y$ variables follows the same argument.
\hfill\ensuremath{\square}

\begin{theorem}[\textbf{Optimality Gap of Algorithm 1}]\label{thm:errorbound}
	Let $OPT$ be the optimal solution to the SBSP and let $z_{round}$ be the solution from Algorithm 1. We have
	\[z_{round}\leq OPT+\sqrt{2R_{max}\mathrm{log}(2T)OPT}+R_{max}\mathrm{log}(2T)\]
	with probability at least $\frac{1}{2}$, where $R_{max}$ is the maximum number of routes in one school.
\end{theorem}
Theorem \ref{thm:errorbound} shows that the optimality gap of  the rounding algorithm can be bounded by $\sqrt{2R_{max}\mathrm{log}(2T)OPT}+R_{max}\mathrm{log}(2T)$. In reality, $R_{max}$ is usually bounded by school capacity and the term $\mathrm{log}(2T)$ can be treated as a constant. In the case where the number of routes in each school is uniformly bounded and $\mathrm{log}(2T)$ is (or can be bounded by) a constant, our optimality gap is of order $\sqrt{OPT}$. We also note that the probability $\frac{1}{2}$ can be boosted to $(1-\varepsilon)$ for any $\varepsilon>0$ by running Algorithm 1 $\lceil\mathrm{log}({1/\varepsilon})\rceil$ times and selecting the best solution.

\textbf{Proof of Theorem \ref{thm:errorbound}.}\label{proof:errorbound}
Let $\{(x^{*}_s,y^{*}_s)\}_{s\in \mathcal{S}}$ be an optimal fractional solution to the LP relaxation of \ref{ILP2} and let $\{(x_s,y_s)\}_{s\in \mathcal{S}}$ be the integral solution obtained from Algorithm 1. Let $z^{*}_{LP}$ be the optimal objective value of the LP relaxation of \ref{ILP2} and let $z_{round}$ be the objective value of \ref{ILP2} using the integral solution $\{(x_s,y_s)\}_{s\in \mathcal{S}}$.

For any $t\in [T]$, let $z_t=\sum_{i\in \mathcal{R} }\sum_{t^{'}=t}^{t+r_i-1} x_{i,t^{'} }$.  From Proposition \ref{prop:prob_property},
$$\mathbb{E}[z_t]=\mathbb{E}\Big[\sum_{i\in \mathcal{R} }\sum_{t^{'}=t}^{t+r_i-1} x_{i,t^{'} }\Big]=\sum_{i\in \mathcal{R} }\sum_{t^{'}=t}^{t+r_i-1} x^{*}_{i,t^{'} }\leq z^{*}_{LP}.$$

Note that $z_t=\sum_{i\in \mathcal{R} }\sum_{t^{'}=t}^{t+r_i-1} x_{i,t^{'} }=\sum_{s\in \mathcal{S}}\big(\sum_{i\in \mathcal{R}_s }\sum_{t^{'}=t}^{t+r_i-1} x_{i,t^{'} }\big)$ is the sum of $|\mathcal{S}|$ independent random variables ($x$ variables for different schools are independent) where each of them is in $[0,R_{max}]$. The next step is to bound the deviation of $z_t$ from its mean with high probability for each $t\in [T]$, which can be derived from the Chernoff bound (\cite{chernoff1952measure}).

\begin{lemma}[Chernoff Bound]\label{lemma:chernoff}
	Let $a_1,a_2,\ldots,a_n$ be $n$ independent random variables in $[0,M]$ and $A=\sum_{i=1}^{n}a_i$ with mean $\mathbb{E}[A]\leq \mu$. For any $\varepsilon>0$,
	\[\mathrm{Pr}[A>\mu+\varepsilon]\leq \mathrm{exp}\Big(-\frac{\varepsilon^{2}}{(2\mu+\varepsilon)M} \Big).\]
\end{lemma}

For any $t\in [T]$, note that
$$z_t=\sum_{i\in \mathcal{R} }\sum_{t^{'}=t}^{t+r_i-1} x_{i,t^{'} }=\sum_{s\in \mathcal{S}}\Big(\sum_{i\in \mathcal{R}_s }\sum_{t^{'}=t}^{t+r_i-1} x_{i,t^{'} }\Big)$$
is the sum of $|\mathcal{S}|$ independent random variables where each of them lies in $[0,R_{max}]$. Further,
$$\mathbb{E}[z_t]=\sum_{i\in \mathcal{R} }\sum_{t^{'}=t}^{t+r_i-1} x_{i,t^{'} }^{*}\leq z^{*}_{LP}.$$

Let $\varepsilon^{*}$ be the positive root of $\mathrm{exp}\big(-\frac{(\varepsilon^{*})^{2}}{(2z^{*}_{LP}+\varepsilon^{*})R_{max}}\big)=\frac{1}{2T}$. Using Lemma \ref{lemma:chernoff}, we have
$$\mathrm{Pr}\big[z_t>z^{*}_{LP}+\varepsilon^{*} \big]\leq \frac{1}{2T} $$
for all $t\in [T]$.

Applying a union bound over all $t\in [T]$, we have
$$\mathrm{Pr}\big[\mathrm{max}_{t\in [T]}\{z_t\}>z^{*}_{LP}+\varepsilon^{*} \big]\leq T\cdot \frac{1}{2T}= \frac{1}{2}.$$

From constraint \eqref{eqn:2d} in \ref{ILP2},
\begin{align*}
	z_{round}&=\mathrm{max}_{t\in [T]}\Big\{\sum_{i\in \mathcal{R} }\sum_{t^{'}=t}^{t+r_i-1} x_{i,t^{'} } \Big\}=\mathrm{max}_{t\in [T]}\{z_t\}.
\end{align*}

Therefore, with probability at least $\frac{1}{2}$, $z_{round}\leq z^{*}_{LP}+\varepsilon^{*}$.

Finally, it is easy to calculate that
\begin{align*}
	\varepsilon^{*}&=\frac{R_{max}\mathrm{log}(2T)+\sqrt{(R_{max}\mathrm{log}(2T))^{2}+8R_{max}\mathrm{log}(2T)z^{*}_{LP}} }{2}\\
	&\leq\frac{R_{max}\mathrm{log}(2T)+\sqrt{(R_{max}\mathrm{log}(2T))^{2}}+\sqrt{8R_{max}\mathrm{log}(2T)z^{*}_{LP}} }{2}\\
	&=R_{max}\mathrm{log}(2T)+\sqrt{2R_{max}\mathrm{log}(2T)z^{*}_{LP}}.
\end{align*}
Note that $z^{*}_{LP}\leq OPT$, this completes the proof of Theorem \ref{thm:errorbound}.
\hfill\ensuremath{\square}

In fact, the proof of Theorem \ref{thm:errorbound} implies a stronger result, i.e.
\[z_{round}\leq z^{*}_{LP}+R_{max}\mathrm{log}(2T)+\sqrt{2R_{max}\mathrm{log}(2T)z^{*}_{LP}} \]
with probability at least $\frac{1}{2}$. Since $z_{round}\geq OPT$, we have the following corollary on the integrality gap of \ref{ILP2}.
\begin{corollary}[Integrality Gap of \ref{ILP2}]\label{coro:integrality-gap}
	\[z^{*}_{LP}\leq OPT\leq z^{*}_{LP}+R_{max}\mathrm{log}(2T)+\sqrt{2R_{max}\mathrm{log}(2T)z^{*}_{LP}}.\]
\end{corollary}


\section{Robust SBSP} \label{sec:robust}
In practice, the route set of a school district, $\mathcal{R}$, may vary over time due to reasons such as the redesign of bus routes and enrollment changes.  School start times, however, cannot be changed frequently.  Thus, decision makers are interested in school scheduling plans that can work for a long period of time. Motivated by this, we introduce a robust version of the SBSP. To solve the robust SBSP, we generalize results of Sections \ref{sec:ILP} and \ref{sec:rounding} and provide efficient algorithms with provable performance guarantee.

At the time school schedule plans are made, there may be uncertainty in future route sets in terms of the length of routes and the number of routes at each school.  We model these uncertainties in the route set $\mathcal{R}$ with multiple scenarios where each scenario has a set of routes that may differ by number and length of routes.  Scenarios can be derived from historical route data  or from newly designed routes to accommodate projected future enrollment distributions. Our goal is to determine school start times and route schedules for each scenario such that all scenarios share the same school start times. The objective is to minimize the maximum number of buses over all scenarios. We show that the time-indexed formulation and the randomized rounding algorithm can be naturally adapted to solve the robust SBSP and provide similar theoretical optimality gaps.

Let $\Pi$ be the set of scenarios. In each scenario $\pi\in \Pi$, $\mathcal{R}^{\pi}$ is the route set, $\mathcal{R}^{\pi}_s$ is the set of routes for school $s$ and $r^{\pi}_{i}$ is the length of route $i$. The robust SBSP can be modeled as follows.

\begin{align} 
	\textrm{min}~~&z_{robust} \label{ILP3} \tag{ILP-robust} \\
	\textrm{s.t.}~~&\sum_{t=1}^{T} x^{\pi}_{i,t}=1 & \forall i\in \mathcal{R}, \pi\in \Pi \label{eqn:3a} \tag{3a} \\
	&\sum_{t=1}^{T}y_{s,t}=1 &\forall s\in \mathcal{S} \label{eqn:3b} \tag{3b} \\
	&\sum_{t^{'}=1}^{t} x^{\pi}_{i,t^{'} }\leq \sum_{t^{'}=1}^{t+l_s} y_{s,t^{'} } &\forall i\in \mathcal{R}^{\pi}_s, s\in \mathcal{S}, t\in [T], \pi\in \Pi \label{eqn:3c} \tag{3c} \\
	&\sum_{t^{'}=1}^{t} y_{s,t^{'} }\leq \sum_{t^{'}=1}^{t} x^{\pi}_{i,t^{'} } &\forall i\in \mathcal{R}^{\pi}_s, s\in \mathcal{S}, t\in [T], \pi\in \Pi \label{eqn:3d} \tag{3d} \\
	&\sum_{i\in \mathcal{R} }\sum_{t^{'}=t}^{t+r^{\pi}_{i}-1} x^{\pi}_{i,t^{'} }\leq z_{robust} &\forall t\in [T], \pi\in \Pi \label{eqn:3e} \tag{3e} \\
	&x^{\pi}_{i,t}\in \{0,1\} &\forall i\in \mathcal{R}^{\pi}, t\in [T], \pi\in \Pi \label{eqn:3f} \tag{3f}\\
	&y_{s,t}\in \{0,1\} &\forall s\in \mathcal{S}, t\in [T] \label{eqn:3g} \tag{3g}
\end{align}

\ref{ILP3} can be viewed as an extension of \ref{ILP2}. In \ref{ILP3}, binary variable $x^{\pi}_{i,t}$ indicates the arrival time of route $i$ in scenario $\pi$. Since all scenarios share the same school start times, we use binary variable $y_{s,t}$ to indicate the start time of school $s$ across all scenarios. Similar to \ref{ILP2}, \ref{ILP3} consists of assignment constraints \eqref{eqn:3a}-\eqref{eqn:3b}, time-window constraints \eqref{eqn:3c}-\eqref{eqn:3d}, linking constraints \eqref{eqn:3e} and binary constraints \eqref{eqn:3f}-\eqref{eqn:3g}. The objective is minimize the maximum number of buses over all scenarios.

Similar to Algorithm 1, after solving the LP relaxation of \eqref{ILP3}, for any school $s$, the start time of $s$ and arrival times of routes in $s$ over all scenarios are determined simultaneously (through the rounding algorithm) to ensure feasibility. Moreover, we use the same probabilistic approach (detailed in Theorem \ref{thm:errorbound}) to provide an optimality gap of the rounding algorithm for the robust SBSP.

\begin{table}[h]\label{table:compare}
	\begin{center}
		\begin{tabular}{ |c|c|c| }
			\hline
			Number of scenarios & single ($|\Pi|=1$) & multiple ($|\Pi|>1$) \\ \hline
			Number of ILP variables & $O\big(T(|\mathcal{S}|+|\mathcal{R}|)\big)$ & $O\big(T(|\mathcal{S}|+|\mathcal{R}|\cdot|\Pi|)\big)$ \\ \hline
			Number of ILP constraints & $O\big(T(|\mathcal{S}|+|\mathcal{R}|)\big)$ & $O\big(T(|\mathcal{S}|+|\mathcal{R}|\cdot|\Pi|)\big)$ \\ \hline
			Optimality gap of rounding algorithm & $O(\sqrt{R_{max}\mathrm{log}(T)OPT})$ & $O(\sqrt{R_{max}\mathrm{log}(|\Pi|\cdot T)OPT})$ \\ \hline
		\end{tabular}
	\end{center}
	\caption{Comparison between problems with single and multiple scenarios}
\end{table}

Table 1 compares the formulations and solution approaches for the SBSP and the robust SBSP in terms of formulation size and theoretical optimality gap. Compared to \ref{ILP2}, the number of variables and constraints in \ref{ILP3} increases linearly with the number of scenarios and the optimality gap of the rounding algorithm includes an additional term $\sqrt{\mathrm{log(|\Pi|)}}$. In practice, this additional term $\sqrt{\mathrm{log(|\Pi|)}}$ can always be viewed as a constant; the rounding algorithm still provides near-optimal solutions for large-scale instances.


\section{Numerical Study} \label{sec:numerical}
In this section, we conduct numerical experiments on three types of SBSP instances to complement the theoretical findings and generalize these findings to incorporate route transition times.

\begin{itemize}
	\item \textbf{Instances with zero transition time.} We test the rounding algorithm on instances with zero transition time to (i) demonstrate the power of the strengthened formulation (ILP2) on providing tight lower bounds, and (ii) show the near-optimality of the rounding algorithm on large-scale instances.
	
	\item \textbf{Instances with non-zero transition time.} Although the randomized rounding algorithm and its theoretical performance guarantee rely on the assumption of zero transition time, we present a modification of the algorithm which provides both feasible solutions and cost lower bounds for instances with non-zero transition time. This allows us to solve instances with location-dependent transition time.
	
	\item \textbf{Instances with  multiple scenarios, i.e., the robust SBSP.} We illustrate the capability of our approach on providing both cost lower bounds and upper bounds (feasible solutions) when there is uncertainty in the route set $\mathcal{R}$. 
\end{itemize}

The remainder of this section is organized as follows: in Section \ref{subsec:synthetic}, we explain the synthetic data we use to generate all SBSP instances. In Sections \ref{subsec:numerical-zero} and \ref{subsec:numerical-non-zero}, we study the performance of the rounding algorithm on SBSP instances with zero and non-zero transition time, respectively. Finally, in Section \ref{subsec:robust_sbsp}, we provide numerical results for the robust SBSP.

\subsection{Synthetic Data and Experimental Design}\label{subsec:synthetic}
We generate ten SBSP instances where the number of schools and routes are $(10,50), (20,100), \cdots,  (100,500)$. Across all instances, each school has a fixed time window of 20 minutes. School start times and route arrival times can be chosen from $\{1,2,\cdots,120\}$, which corresponds to a 2-hour time window. School start times are further restricted to multiples of 5; i.e. $\{5,10,\cdots,120\}$. The route set for each instance is generated as follows.

\begin{itemize}
	\item Step 1: Generate a 100$\times$100 grid and locate schools on integral points of the grid at random
	
	\item Step 2: For each route, select a random integral point of the grid as its starting point and a random school location as its ending point
	
	\item Step 3: The length of each route is computed by the $L_{1}$ distance between its starting point and ending point (which is a school), divided by a constant travel speed $v_1$
	
	\item Step 4: Round route lengths to the nearest integer
\end{itemize}

The transition time between each pair of routes is computed by the $L_{1}$ distance between the ending point of the first route and the starting point of the second route, divided by a constant transition speed $v_2$.

For travel speed  $v_1$, we select a constant such that the average route length is 30 minutes. For transition speed $v_2$, we use different parameters to generate instances with zero and non-zero transition time. For instances with zero transition time, we let $v_2$ be a large constant. For instances with non-zero transition time, we select a constant $v_2$ such that the average transition time is approximately 15 minutes.

Following the above steps, we generate ten SBSP instances with zero transition time and ten with non-zero transition time. In each instance we perform the randomized rounding algorithm 10 times (after solving the LP relaxation) and select the best solution. In the following subsections, we analyze the performance of the randomized rounding algorithm on both types of instances by comparing (all or part of) the following metrics. 
\begin{enumerate}
	\item The exact solution obtained by solving the ILP (only for instances with zero transition time)
	\item Lower bounds obtained from LP relaxations of ILP1 and ILP2
	\item Upper bound obtained from the randomized rounding algorithm
	\item Improved upper bound by combining the rounding algorithm with a local search heuristic
	\item Upper bound obtained from only a local search heuristic
\end{enumerate}

With these metrics, we compare the exact solution (for instances with zero transition time), cost lower bounds and upper bounds to evaluate the quality of solutions obtained from the rounding algorithm.

Local search heuristics have been widely applied to scheduling problems due to their simplicity and scalability. In the context of school bus scheduling, local search heuristics have been used to optimize school start times and bus routes (\citet{spada2005decision}, \citet{chen2015exact} and \citet{bertsimas2019optimizing}). In each iteration of these heuristics, a typical strategy is to optimize the objective for a small set of schools, keeping decisions for all other schools fixed.

In each iteration of the local search heuristic, we select one school at random and optimize its start time by enumerating all possible start time choices (and shift arrival times for routes in the selected school to maintain feasibility) and then select the one that requires the minimum number of buses. We apply two policies to initialize the local search heuristic: first, we use the solution of the randomized rounding algorithm (the third metric) as a starting point; second, we draw multiple starting points at random and select the best one. For both policies, we set a 2-hour time limit for local search. For the second policy, we use the first 30 minutes to select starting points which is then followed by a 2-hour local search.

\subsection{Numerical Results on Instances with Zero Transition Time}\label{subsec:numerical-zero}

We first illustrate the strength of our ILP formulations by comparing the lower bounds obtained from the LP relaxations of ILP1 and ILP2 with the exact optimal solution (obtained by solving ILP2). As shown in Table 2, the relaxation of ILP2 outperforms that of ILP1 and gives tight lower bound for all instances. Observe that in all instances, rounding up the lower bound from the LP-relaxation of ILP2 provides the integer optimal value. 

\begin{table}[h!]\label{table:ilp1-ilp2}
	\scriptsize
	\centering
	\begin{tabular}{c|c|c|c} 
		\hline
		Instance size & exact solution & lower bound (relaxation of ILP1) & lower bound (relaxation of ILP2) \\ 
		\hline
		\shortstack{\\10 schools \\ 50 routes} & 9 & 7.7 & 8.5 \\ \hline
		\shortstack{\\20 schools \\ 100 routes} & 17 & 15.6 & 16.5 \\ \hline
		\shortstack{\\30 schools \\ 150 routes} & 24  & 22.0 & 23.9 \\ \hline
		\shortstack{\\40 schools \\ 200 routes} & 32 & 30.8 & 31.1 \\ \hline
		\shortstack{\\50 schools \\ 250 routes} & 42 & 38.1 & 40.0 \\ \hline
		\shortstack{\\60 schools \\ 300 routes} & 51 & 48.1 & 50.9 \\ \hline
		\shortstack{\\70 schools \\ 350 routes} & 61 & 58.2 & 60.5 \\ \hline
		\shortstack{\\80 schools \\ 400 routes} & 65 & 61.5 & 64.5 \\ \hline
		\shortstack{\\90 schools \\ 450 routes} & 76 & 71.7 & 75.2 \\ \hline
		\shortstack{\\100 schools \\ 500 routes} & 84 & 80.6 & 83.8 \\ \hline
	\end{tabular}
	\caption{Comparison between LP relaxations of ILP1 and ILP2}
\end{table}

For larger instances where getting the exact solution is intractable, we compare the lower bounds obtained from ILP1 and ILP2 with the feasible solution from the rounding algorithm. We test the algorithms on three sets of instances with different number of schools and routes (each set with five instances). As shown in Table 3, the relative gap between the rounding solution and the lower bound from ILP2 shrinks quickly as the instance size increases. This supports our theoretical findings that ILP2 is a strong formulation for the SBSP and that the randomized rounding algorithm provides high-quality solution for large-scale instances.

\begin{table}[h!]\label{table:large-zero-tran}
	\centering
	\begin{tabular}{c|c|c} 
		\hline
		Instance size & relative gap of ILP1 & relative gap of ILP2\\
		\hline
		\shortstack{\\200 schools \\ 1000 routes} & 10.7\% & 5.7\% \\ \hline
		\shortstack{\\500 schools \\ 2500 routes} & 8.2\% & 3.7\% \\ \hline
		\shortstack{\\1000 schools \\ 5000 routes} & 7.3\% & 2.7\% \\ \hline
	\end{tabular}
	\caption{Large instances with zero transition time (averaged over 5 instances)}
\end{table}

We next show the efficiency of the randomized rounding algorithm by comparing it with other heuristics. Table 4 compares the exact solution, the lower bound obtained from the LP relaxation of ILP2 and three upper bounds (rounding, rounding + local search, and local search) for each SBSP instance.

The randomized rounding algorithm achieves an average optimality gap of $12.9\%$ which is further improved to $10.3\%$ with local search. In contrast, the local search heuristic yields an optimality gap of $37.3\%$. The randomized rounding algorithm outperforms the local search heuristic for all instances in terms of solution quality. Moreover, the rounding algorithm solves the largest instance within 2 minutes while the local search heuristic reaches the 2-hour time limit for the same instance. For all 10 instances, we are able to solve the ILP to optimality (using ILP2) within 60 minutes. Our numerical results show that the randomized rounding algorithm is efficient in practice and provides better solutions compared to the local search heuristic.

\begin{table}[h!]\label{table:zero-tran}
	\scriptsize
	\centering
	\begin{tabular}{c|c|c|c|c|c} 
		\hline
		Instance size & exact solution & lower bound (ILP2) & rounding & rounding + local search & local search \\ 
		\hline
		\shortstack{\\10 schools \\ 50 routes} & 9 & 9 & 11 & 11 & 12 \\ \hline
		\shortstack{\\20 schools \\ 100 routes} & 17 & 17 & 21 & 19 & 24 \\ \hline
		\shortstack{\\30 schools \\ 150 routes} & 24 & 24 & 27 & 27 & 32  \\ \hline
		\shortstack{\\40 schools \\ 200 routes} & 32 & 32 & 36 & 35 & 45  \\ \hline
		\shortstack{\\50 schools \\ 250 routes} & 42 & 41 & 46 & 46 & 56 \\ \hline
		\shortstack{\\60 schools \\ 300 routes} & 51 & 51 & 57 & 55 & 66  \\ \hline
		\shortstack{\\70 schools \\ 350 routes} & 61 & 61 & 66 & 65 & 89 \\ \hline
		\shortstack{\\80 schools \\ 400 routes} & 65 & 65 & 70 & 69 & 92  \\ \hline
		\shortstack{\\90 schools \\ 450 routes} & 76 & 76 & 81 & 80 & 101  \\ \hline
		\shortstack{\\100 schools \\ 500 routes} & 84 & 84 & 94 & 92 & 116 \\ \hline
		\textbf{Optimality gap} & \textbf{0\%} &  & \textbf{12.9\%} & \textbf{10.3\%} & \textbf{37.3\%} \\ \hline
	\end{tabular}
	\caption{Solution comparison for instances with zero transition time}
\end{table}

As noted in the introduction, transportation costs are just a part of the consideration in determining school start times.  Therefore, it is valuable to provide school districts with a set of school start time options that yield comparable transportation costs.  In our experiments, we record all solutions obtained with the rounding algorithm, in addition to the best solutions reported in Table 4.  We find that all generated solutions are within 10\% of the best solution and 73.3\% of solutions are within 5\% of the best solution. Therefore, the rounding algorithm is able to generate multiple high quality solutions efficiently. 

\subsection{Numerical Results on Instances with Non-zero Transition Time}\label{subsec:numerical-non-zero}
We modify our randomized rounding algorithm to accommodate different non-zero transition times and test its performance in numerical experiments. 
\subsubsection{A Modified Randomized Rounding Algorithm}
We have shown the efficiency of the rounding algorithm in solving SBSP instances with zero transition time, as well as its near-optimality for large-scale instances. However, the assumption of zero transition time does not hold in all settings. We provide a modification of the rounding algorithm which is able to provide both feasible solutions and cost lower bounds for instances with non-zero transition time through the following steps.

\begin{itemize}
	\item Step 1: reducing the problem to a tractable instance
	\item Step 2: determining route arrival times using the randomized rounding algorithm
	\item Step 3: obtaining a feasible solution for the original problem
	\item Step 4: obtaining a cost lower bound
\end{itemize}

\underline{\textbf{Step 1: reducing a SBSP instance to a tractable instance}}

We first define a class of SBSP instances with non-zero transition time that are tractable using the rounding algorithm.

\begin{definition}(Tractable instance)
	A SBSP instance is tractable if there exists $t_{i}^{in}$ and $t_{i}^{out}$ for each route $i$ such that the transition time from route $i$ to route $j$ is equal to $t_{i}^{out}+t_{j}^{in}$ for all $(i,j)$.
\end{definition}

For a tractable instance, the transition time $t_{ij}$ can be decomposed into two parts -- an ``outbound transition time" $t_{i}^{out}$ and an ``inbound transition time" $t_{j}^{in}$. By adding them into the route length (i.e. a route with length $r_i$ becomes one with length $t_{i}^{in}+r_i+t_{i}^{out}$), we can further reduce it to an instance with zero transition time. Therefore, the randomized rounding algorithm can directly be applied to solve a tractable SBSP instance.

Given a SBSP instance with non-zero transition time, our goal is to first find a tractable instance that approximates this instance. In other words, we find $t_{i}^{in}$ and $t_{i}^{out}$ such that $\textrm{tran}_{ij}$ (transition time in the given instance) can be approximated by $t_{i}^{out}+t_{j}^{in}$. To that end, we fit $t_{i}^{in}$ and $t_{i}^{out}$ into the following linear regression model.

\begin{equation}\label{eqn:regression}
	\textrm{tran}_{ij}=t_{i}^{out}+t_{j}^{in}+\varepsilon_{ij}
\end{equation}
In practice, we use the least square regression (i.e. minimizing $\sum \varepsilon_{ij}^{2}$) to obtain $t_{i}^{in}$ and $t_{i}^{out}$.

\underline{\textbf{Step 2: determining route arrival times}}

Having obtained a tractable approximation from Step 1, we reduce it to an instance with zero transition time. In Step 2, we apply the randomized rounding algorithm to solve the tractable instance (obtained from Step 1) and determine the start time of each school and the arrival time of each route for the original problem.

\underline{\textbf{Step 3: obtaining a feasible solution}}

After determining the arrival time for each route, we find the minimum number of buses to serve all the routes using the following proposition.

\begin{proposition}\label{prop:bipartite-matching}
	Given a route set $\mathcal{R}$ and an arrival time of each route, consider a bipartite graph $G=(U,V; E)$ where $U=V=\mathcal{R}$. For any $r_1\in U$ and $r_2\in V$, $(r_1,r_2)\in E$ if and only if route $r_2$ can be served after route $r_1$ using the same bus. Let $|G_{BM}|$ be the size of the maximum bipartite matching of graph $G$. Then, the minimum number of buses to complete all routes is equal to $|\mathcal{R}|-|G_{BM}|$.
\end{proposition}

From Proposition \ref{prop:bipartite-matching}, finding the minimum number of buses is equivalent to solving a maximum bipartite matching problem, which can be done in polynomial time. To see the correctness of Proposition \ref{prop:bipartite-matching}, each route connection pattern corresponds to a bipartite matching of graph $G$ where $(r_1,r_2)\in E$ if and only if route $r_2$ is served immediately after route $r_1$ using the same bus. Moreover, the size of this bipartite matching is equal to $|\mathcal{R}|$ minus the number of buses. Therefore, minimizing the number of buses is equivalent to finding a maximum bipartite matching on $G$.

\underline{\textbf{Step 4: obtaining a cost lower bound}}

To get a lower bound on the number of buses, we enforce the error term $\varepsilon_{ij}$ in the regression model \eqref{eqn:regression} to be non-negative so that the estimated transition time $t_{i}^{out}+t_{j}^{in}$ is less than or equal to $\textrm{tran}_{ij}$. By inserting the $t_{i}^{in}$ and $t_{i}^{out}$ into ILP2 and its LP relaxation, we are able to obtain a lower bound on the number of buses.

\subsubsection{Numerical Results}
We compare different algorithms on SBSP instances with non-zero transition time. Since we are not able to compute the exact solution, we compare the lower bound obtained using the modified rounding algorithm and the upper bounds (feasible solutions) obtained from the rounding heuristic, the local search heuristic, and the combined heuristic.

As shown in Table 5, the optimality gap of the rounding algorithm is $37.1\%$ which can be improved to $35.4\%$ with local search heuristic. Local search with random restarts yields an optimality gap of $55.6\%$. The optimality gap of the rounding algorithm grows quickly compared to instances with zero transition time. One possible reason is that transition times are under-estimated in the regression model in Step 1 so that we are not able to get a strong lower bound as we did for instances with zero transition time. Despite that, the rounding algorithm provides near locally-optimal solutions and outperforms the local search heuristic in terms of solution quality. Although we believe the large optimality gap mainly comes from the lower bound side, it remains unknown whether the non-zero transition time can be better incorporated into the rounding algorithm, which is a potential topic for future research.

\begin{table}[h!]\label{table:nonzero-tran}
	\centering
	\begin{tabular}{c|c|c|c|c} 
		\hline
		Instance size & lower bound (ILP2) & rounding & rounding + local search & local search \\ 
		\hline
		\shortstack{\\10 schools \\ 50 routes} & 10 & 15 & 15 & 15 \\ \hline
		\shortstack{\\20 schools \\ 100 routes} & 19 & 26 & 26 & 31 \\ \hline
		\shortstack{\\30 schools \\ 150 routes} & 26 & 37 & 37 & 42 \\ \hline
		\shortstack{\\40 schools \\ 200 routes} & 37 & 50 & 48 & 57 \\ \hline
		\shortstack{\\50 schools \\ 250 routes} & 45 & 59 & 59 & 68 \\ \hline
		\shortstack{\\60 schools \\ 300 routes} & 52 & 70 & 70 & 77 \\ \hline
		\shortstack{\\70 schools \\ 350 routes} & 63 & 85 & 83 & 101 \\ \hline
		\shortstack{\\80 schools \\ 400 routes} & 67 & 93 & 90 & 102 \\ \hline
		\shortstack{\\90 schools \\ 450 routes} & 75 & 101 & 100 & 123 \\ \hline
		\shortstack{\\100 schools \\ 500 routes} & 86 & 114 & 112 & 130 \\ \hline
		\textbf{Average gap} & \textbf{0\%} & \textbf{37.1\%} & \textbf{35.4\%} & \textbf{55.6\%} \\ \hline
	\end{tabular}
	\caption{Instances with non-zero transition time}
\end{table}

\subsection{Numerical Study for the Robust SBSP}\label{subsec:robust_sbsp}
We construct instances for the robust SBSP from the 20 instances (ten with zero transition time and ten with non-zero) generated above. We expand each of the 20 instances to five scenarios by incorporating two types of uncertainty into the route set $\mathcal{R}$: (i) uncertainty in the number of routes, and (ii) uncertainty in travel times. For uncertainty (i), we increase or decrease the number of routes in each school by 1, each with probability $15\%$. For uncertainty (ii), we add a random integer from $[-5,5]$ into the length of each route. The transition times are perturbed in a similar way. We construct three groups of instances where each group incorporates one or both types of uncertainty.

For each instance in each group, we compare the cost lower bound and the cost upper bound obtained from the modified rounding algorithm. Tables 6-8 summarize result for each group of instances. Compared to Tables 4 and 5, the average relative gap of the rounding algorithm increases from $12.9\%$ to $24.1\%\sim 26.7\%$ for instances with zero transition time, and from $37.1\%$ to $42.7\%\sim 47.1\%$ for instances with non-zero transition time. We point out the number of buses can be highly sensitive to small changes in route lengths and number of routes (e.g. adding a new route may require a new bus). As a result, it is harder to get tight lower bounds for the robust SBSP by combining information from different scenarios.

For instances with zero transition time, we observe that the relative gap between the upper and lower bounds shrinks as instance size goes up. This shows that the formulation \ref{ILP3} and the modified randomized rounding algorithm remains effective for the robust SBSP, especially for large-scale instances. For instances with non-zero transition time, our approach is capable of providing both feasible solutions and cost lower bounds without increasing the relative gaps by too much. From a strategic planning perspective, our approach provides an efficient way to identify scheduling plans with sufficient potential for cost savings (and remove from consideration those without sufficient savings) before continuing with more detailed operational analysis.

\begin{table}\label{table:robust-num}
	\centering
	\begin{tabular}{c|c|c|c|c}
		\hline
		Instance size &
		\multicolumn{2}{c}{zero transition} &
		\multicolumn{2}{c}{non-zero transition}  \\ \hline
		
		& lower bound & rounding & lower bound & rounding \\ \hline
		\shortstack{\\10 schools \\ 50 routes} & 9 & 12 & 10 & 12 \\ \hline
		\shortstack{\\20 schools \\ 100 routes} & 16 & 22 & 18 & 25 \\ \hline
		\shortstack{\\30 schools \\ 150 routes} & 24 & 33 & 27 & 41 \\ \hline
		\shortstack{\\40 schools \\ 200 routes} & 33 & 45 & 35 & 54 \\ \hline
		\shortstack{\\50 schools \\ 250 routes} & 39 & 50 & 43 & 62 \\ \hline
		\shortstack{\\60 schools \\ 300 routes} & 48 & 59 & 53 & 78 \\ \hline
		\shortstack{\\70 schools \\ 350 routes} & 56 & 66 & 60 & 86 \\ \hline
		\shortstack{\\80 schools \\ 400 routes} & 66 & 78 & 69 & 98 \\ \hline
		\shortstack{\\90 schools \\ 450 routes} & 71 & 80 & 74 & 107 \\ \hline
		\shortstack{\\100 schools \\ 500 routes} & 81 & 91 & 84 & 121 \\ \hline
		\textbf{Average gap} &  & \textbf{25.7\%} &  & \textbf{42.7\%} \\ \hline
	\end{tabular}
	\caption{Robust instances with uncertainty in number of routes}
\end{table}

\begin{table}\label{table:robust-length}
	\centering
	\begin{tabular}{c|c|c|c|c}
		\hline
		Instance size &
		\multicolumn{2}{c}{zero transition} &
		\multicolumn{2}{c}{non-zero transition}  \\ \hline
		
		& lower bound & rounding & lower bound & rounding \\ \hline
		\shortstack{\\10 schools \\ 50 routes} & 9 & 12 & 10 & 13 \\ \hline
		\shortstack{\\20 schools \\ 100 routes} & 17 & 23 & 19 & 30 \\ \hline
		\shortstack{\\30 schools \\ 150 routes} & 26 & 33 & 27 & 39 \\ \hline
		\shortstack{\\40 schools \\ 200 routes} & 34 & 43 & 36 & 51 \\ \hline
		\shortstack{\\50 schools \\ 250 routes} & 42 & 55 & 44 & 66 \\ \hline
		\shortstack{\\60 schools \\ 300 routes} & 51 & 64 & 52 & 76 \\ \hline
		\shortstack{\\70 schools \\ 350 routes} & 57 & 67 & 62 & 87 \\ \hline
		\shortstack{\\80 schools \\ 400 routes} & 65 & 77 & 72 & 103 \\ \hline
		\shortstack{\\90 schools \\ 450 routes} & 72 & 82 & 79 & 114 \\ \hline
		\shortstack{\\100 schools \\ 500 routes} & 82 & 92 & 85 & 125 \\ \hline
		\textbf{Average gap} &  & \textbf{24.1\%} &  & \textbf{44.2\%} \\ \hline
	\end{tabular}
	\caption{Robust instances with uncertainty in travel times}
\end{table}

\begin{table}\label{table:robust-both}
	\centering
	\begin{tabular}{c|c|c|c|c}
		\hline
		Instance size &
		\multicolumn{2}{c}{zero transition} &
		\multicolumn{2}{c}{non-zero transition}  \\ \hline
		
		& lower bound & rounding & lower bound & rounding \\ \hline
		\shortstack{\\10 schools \\ 50 routes} & 8 & 12 & 10 & 14 \\ \hline
		\shortstack{\\20 schools \\ 100 routes} & 17 & 25 & 19 & 29 \\ \hline
		\shortstack{\\30 schools \\ 150 routes} & 25 & 33 & 27 & 41 \\ \hline
		\shortstack{\\40 schools \\ 200 routes} & 32 & 38 & 34 & 49 \\ \hline
		\shortstack{\\50 schools \\ 250 routes} & 41 & 53 & 44 & 66 \\ \hline
		\shortstack{\\60 schools \\ 300 routes} & 46 & 60 & 51 & 76 \\ \hline
		\shortstack{\\70 schools \\ 350 routes} & 54 & 65 & 58 & 86 \\ \hline
		\shortstack{\\80 schools \\ 400 routes} & 66 & 76 & 69 & 103 \\ \hline
		\shortstack{\\90 schools \\ 450 routes} & 73 & 82 & 80 & 113 \\ \hline
		\shortstack{\\100 schools \\ 500 routes} & 84 & 94 & 87 & 126 \\ \hline
		\textbf{Average gap} &  & \textbf{26.7\%} &  & \textbf{47.1\%} \\ \hline
	\end{tabular}
	\caption{Robust instances with uncertainty in both number of routes and travel times}
\end{table}


\section{Discussion} \label{sec:conclusion}
In this paper, we develop a time-indexed ILP formulation for the SBSP and use its LP relaxation to develop a randomized rounding algorithm that yields near-optimal solutions for large-scale instances in polynomial time. Our methodologies are also adapted to solve a robust version of the SBSP. Our computational study suggests that the rounding algorithm is efficient and effective to solve problems where the transition time between routes is a constant. Our approach is generalized to instances with non-constant transition times as well as to obtain robust solutions for multiple scenarios.

In this work, we assume that bus routes are taken as inputs in the scheduling problem. For future research, we would like to build a unified framework that jointly solves route design problem and school bus scheduling problem. Another important question is to integrate our results with other strategic decisions, such as school boundary design and school location selection. Our approach can serve as a toolbox to help decision makers understand the effect on school transportation when making policy changes. Both generalizations capture real life problems, and we will continue working with our partner school district to provide solution approaches that are efficient, robust and easy to implement.
\section*{Acknowledgement}
This project is supported by the National Science Foundation (CMMI-1727744). The authors thank the administrative staff of Evanston/Skokie Public School District 65.

\newpage
\setcounter{page}{1}
\small

\bibliographystyle{plainnat}
\bibliography{Literature_scheduling}

\end{document}